\documentclass[12pt,reqno]{amsart}
\usepackage{indentfirst, amssymb, amsmath, amsthm, mathrsfs, setspace, indentfirst, enumerate,  mathrsfs, amsmath, amsthm}
\usepackage[bookmarksnumbered, colorlinks, plainpages]{hyperref}
\usepackage{mathrsfs}
\usepackage{cite}
\usepackage{tikz}
\usepackage{graphicx}
\usetikzlibrary{arrows.meta,calc,decorations.pathreplacing,positioning}
\usetikzlibrary{shapes.geometric}
\usepackage{float}
\usepackage{pgfplots}
\usetikzlibrary{shadows}
\pgfplotsset{compat=1.18}
\usetikzlibrary{arrows.meta,positioning,shapes.geometric}
\usepgfplotslibrary{fillbetween}
\definecolor{curveblue}{RGB}{30,90,220}
\definecolor{negred}{RGB}{255,182,193}
\definecolor{posgreen}{RGB}{40,150,70}
\definecolor{rootdark}{RGB}{120,40,20}
\usepgfplotslibrary{fillbetween}
\definecolor{c1}{RGB}{70,40,200}

\definecolor{c2}{RGB}{213,14,0}

\definecolor{c3}{RGB}{0,158,10}

\definecolor{c4}{RGB}{230,126,34}
\newtheorem*{theoA}{Theorem A}
\newtheorem*{theoB}{Theorem B}
\newtheorem*{theoC}{Theorem C}

\newtheorem*{BC}{Bieberbach's conjecture}

\newtheorem*{conjA}{Conjecture A}

\newtheorem*{defiA}{Definition A}
\newtheorem*{defiB}{Definition B}
\newtheorem*{defiC}{Definition C}
\newtheorem*{cor A}{Corollary A}
\newtheorem*{cor B}{Corollary B}

\newtheorem{theo}{Theorem}[section]
\newtheorem{lem}{Lemma}[section]

\newcommand{\be}{\begin{equation}}
	\newcommand{\ee}{\end{equation}}
\newcommand{\beas}{\begin{eqnarray*}}
	\newcommand{\eeas}{\end{eqnarray*}}
\newcommand{\bea}{\begin{eqnarray}}
	\newcommand{\eea}{\end{eqnarray}}

\numberwithin{equation}{section}

\begin{document}
\title[S\MakeLowercase {harp Robertson-Type Coefficient Estimates for Radial}......]{\LARGE S\Large\MakeLowercase {harp} R\MakeLowercase {obertson-Type  Coefficient Estimates for Radial	Starlike and Quasi-Convex of type} B \MakeLowercase {Mappings of order $\alpha$ in Complex}  B\MakeLowercase {anach Spaces}}

\date{}
\author[S. M\MakeLowercase{ajumder}, A. B\MakeLowercase{anerjee} \MakeLowercase{and} G. H\MakeLowercase{aldar}]{S\MakeLowercase{ujoy} M\MakeLowercase{ajumder}$^{1}$, A\MakeLowercase{bhijit} B\MakeLowercase{anerjee}$^2$ \MakeLowercase{and}  G\MakeLowercase{outam} H\MakeLowercase{aldar}$^3$$^*$ }

\address{$^{1}$ Sujoy Majumder, Department of Mathematics, Raiganj University, Raiganj, West Bengal-733134, India.}
\email{sm05math@gmail.com, sjm@raiganjuniversity.ac.in}
\address{$^{2}$ Abhijit Banerjee, Department of Mathematics, University of Kalyani, West Bengal 741235, India.}
\email{abanerjee\_kal@yahoo.co.in, abhijitbanerjee@klyuniv.ac.in}
\address{$^3$ Goutam Haldar, Department of Mathematics, Ghani Khan Choudhury Institute of Engineering and Technology, Narayanpur, Malda-732141, West Bengal, India.}
\email{goutamiitm@gmail.com, goutamiit1986@gmail.com}

\renewcommand{\thefootnote}{}
\footnote{*\emph{Corresponding Author}: Goutam Haldar.}

\renewcommand{\thefootnote}{\arabic{footnote}}
\setcounter{footnote}{0}

\begin{abstract}
In this paper, we study biholomorphic mappings $F: \mathbb{B} \to X$ of the form $F(z) = f(z)z$ for $z \in \mathbb{B}$, where $f: \mathbb{B} \to \mathbb{C}$ is a holomorphic function satisfying $f(0) = 1$. Our main contributions are as follows: We establish sharp estimates for the coefficients of
starlike functions of order $\alpha$ as well as quasi-convex functions of type $B$ and order $\alpha$ by employing the notion of majorization.
The results established here directly extend the results due to Robertson \cite{Robertson-1936} to the broader setting of complex Banach spaces.

\end{abstract}
\thanks{Typeset by \AmS -\LaTeX}
\maketitle
\section{\bf Introduction}
Let $\mathcal{S}(\mathbb{U})$ denote the class of functions $f$ that are univalent in the open unit disk $\mathbb{U} = \{z \in \mathbb{C} : \vert{}z\vert{} < 1\}$, normalized by $f(0) = 0$ and $f'(0) = 1$. Any function $f \in \mathcal{S}(\mathbb{U})$ has a power series expansion of the form:
\begin{align}\label{eq1}
	\displaystyle
	f(z) = z + \sum_{n=2}^{\infty} a_n z^n, \quad z \in \mathbb{U}.
\end{align}

By employing the Area Theorem, Bieberbach \cite{Bieberbach-1916} established that if $f \in \mathcal{S}(\mathbb{U})$, then $\lvert a_2 \rvert \leq 2$. The following theorem, introduced in 1916, laid the foundation for the celebrated \textbf{Bieberbach conjecture}.

\begin{theoA}{\em{\cite{Bieberbach-1916}}} If $f \in \mathcal{S}(\mathbb{U})$, then $|a_2|\leq 2$. Equality holds if, and only if, $f$ is a rotation of the Koebe function.
\end{theoA}

{\bf{Bieberbach's conjecture}} served as the catalyst for the study of univalent functions, driving its development over the last century into the extensive field of mathematical research we see today.
\begin{BC} If $f \in\mathcal{S}(\mathbb{U})$ and is given by \eqref{eq1}, then $|a_n|\leq n$ for any integer $n\geq 2$. Equality holds if and only if $f$ is a rotation of the Koebe function.
\end{BC}

This conjecture evolved into one of the most celebrated problems in twentieth century mathematics and it is now called de Branges theorem (see \cite{Branges-1985}). Unfortunately, Cartan \cite{Cartan-1933} demonstrated that this result does not extend to several complex variables. Consequently, imposing additional geometric properties—such as starlikeness or convexity—on mappings within a given class becomes essential to establish analogous estimates. However, regarding the bounds on homogeneous expansions for normalized biholomorphic starlike mappings on the Euclidean unit ball $\mathbb{B}^n \subset \mathbb{C}^n$, Roper and Suffridge \cite{Roper-Suffridge-1999} constructed a counterexample showing that these analogous estimates fail even for $m = 2$.\vspace{2mm}

A normalized univalent function $f \in \mathcal{S}(\mathbb{U})$ is said to be starlike of order $\alpha\;(0\leq \alpha<1)$ if and only if $\Re\left\lbrace z f'(z)/f(z) \right\rbrace>\alpha$, $z \in \mathbb{U}$.
The subclass consisting of all such starlike functions of order $\alpha$ is denoted by $\mathcal{S}^*_{\alpha}(\mathbb{U})$. Let $\mathcal{K}_{\alpha}(\mathbb{U})$ denote the familiar subclass of $\mathcal{S}(\mathbb{U})$ consisting of functions which are convex of order $\alpha\;(0\leq \alpha<1)$ in $\mathbb{U}$. Then $f\in\mathcal{K}_{\alpha}(\mathbb{U})$ if and only if $\Re\left\lbrace 1+z f''(z)/f'(z) \right\rbrace >\alpha$, $z \in \mathbb{U}.$

At first, we recall the following definition.
\begin{defiA} \cite[Definition 5.3.1]{Thomas-Tuneski} For $z \in \mathbb{U}$, let $\phi(z) = \sum_{n=0}^{\infty} \phi_n z^n$ and $\psi(z) = \sum_{n=0}^{\infty} \psi_n z^n$, where $\psi_n \geqslant 0$ for all $n \geqslant 0$. Then $\phi(z)$ is majorized by $\psi(z)$, written as $\phi(z) \ll \psi(z)$, if for $n \geqslant 0$, $|\phi_n| \leqslant \psi_n$.
\end{defiA}

In 1936, Robertson \cite{Robertson-1936} obtained the following sharp estimates for the coefficients of starlike and convex functions of order $\alpha$ by employing the notion of majorization:
\begin{theoB}\cite{Robertson-1936} Let $f(z)$ be given by $f(z) = z + \sum_{n=2}^{\infty} a_n z^n$. If $f\in\mathcal{S}^*_{\alpha}(\mathbb{U})$, then for $n\geq 2$, we have the following sharp coefficient estimates
\begin{align*}
|a_n|\leq \frac{1}{(n-1)!}\prod\limits_{k=2}^n(k-2\alpha).
\end{align*}
\end{theoB}

\begin{theoC}\cite{Robertson-1936} Let $f(z)$ be given by $f(z) = z + \sum_{n=2}^{\infty} a_n z^n$. If $f\in\mathcal{K}_{\alpha}(\mathbb{U})$, then for $n\geq 2$, we have the following sharp coefficient estimates
\begin{align*}
|a_n|\leq \frac{1}{n!}\prod\limits_{k=2}^n(k-2\alpha).
\end{align*}
\end{theoC}
\vspace{1.2mm}

Before proceeding further, let us define some recurring notation and terminology.

Let $X$ be a complex Banach space with norm $\Vert{}\cdot\Vert{}$. For $r > 0$, the open ball centered at the origin with radius $r$ is defined as $\mathbb{B}_r = \{z \in X : \Vert{}z\Vert{} < r\}$. For $r = 1$, we simply write $\mathbb{B}$. When $X = \mathbb{C}$, the unit disk $\mathbb{B}$ is denoted by $\mathbb{U}$. The boundary and closure of $\mathbb{B}$ are denoted by $\partial \mathbb{B}$ and $\overline{\mathbb{B}}$, respectively.

Now let $X$ and $Y$ be two complex Banach spaces with respect to the norms $||.||_1$ and $||.||_2$ respectively. For simplicity, we denote both norms by $||.||$, when there is no possibility of confusion. Let $L(X, Y)$ denote the set of continuous linear operators from a complex Banach space $X$ into a complex Banach space $Y$ with the standard operator norm 
\begin{align*} 
	||A||=\sup\{||A(z)||: ||z||=1,\quad A\in L(X,Y)\}.
\end{align*} 

Let $I$ be the identity operator in $L(X, X)$. For each $z \in X \setminus \{0\}$, define the set  
\begin{align*}
	\displaystyle 
	T(z) = \{ l_z \in L(X, \mathbb{C}) : l_z(z) = \Vert{}z\Vert{}, \quad \Vert{}l_z\Vert{} = 1 \}.
\end{align*}

This set is guaranteed to be nonempty by the Hahn-Banach theorem. 
Let $\Omega$ be a domain in $X$ and let $f:\Omega \rightarrow Y$ be a mapping. Then $f$ is called holomorphic on $\Omega$ if for any $z\in\Omega$ there is a mapping $Df(z)\in L(X,Y)$, called the Fr\'echet derivative of $f$ at $z$ such that 
\begin{align*} \lim_{h\to 0} \frac{||f(z+h)-f(z)-Df(z)h||}{||h||}=0,\end{align*} i.e., 
\begin{align*} 
	f(z+h)=f(z)+Df(z)h+o(||h||).
\end{align*}

Let $\mathcal{H}(\Omega, \Omega')$ represent the set of all holomorphic mappings from a domain $\Omega \subset X$ into another domain $\Omega' \subset Y$. We simply write $\mathcal{H}(\Omega, X) = \mathcal{H}(\Omega)$. It is well known that if $f \in \mathcal{H}(\Omega,Y)$, then for all $w$ in some neighborhood of $z \in \Omega$, the function can be expressed as the power series expansion:
\begin{align*}f(w) = \sum_{n=0}^{\infty} \frac{1}{n!} D^n f(z)(w - z)^n,\end{align*} 
where $D^n f(z)$ denotes the $n$-th Fréchet derivative of $f$ at $z$, represented by a bounded symmetric $n$-linear mapping from $X^n$ into $X$. Here, $D^n f(z)(w - z)^n$ is standard notation for evaluating the $n$-linear form at $n$ identical copies of $(w - z)$:
\begin{align*}
	\displaystyle
	D^n f(z)(w - z)^n = D^n f(z)(\underbrace{w - z, w - z, \dots, w - z}_{n \text{ times}}),
\end{align*}
with the convention $D^0 f(z) = f(z)$.
A mapping $f \in \mathcal{H}(\mathbb{B})$ is said to be normalized if $f(0) = 0$ and $Df(0) = I$, where $I$ is the identity operator on $X$.
We say that $f$ is:\\
$\bullet$ biholomorphic if it is a bijection onto its image $f(\mathbb{B})$ with a holomorphic inverse $f^{-1}: f(\mathbb{B}) \to \mathbb{B}$;\\
$\bullet$ locally biholomorphic if its Fréchet derivative $Df(z)$ has a bounded inverse for every $z \in \mathbb{B}$.

The family of all normalized biholomorphic mappings on $\mathbb{B}$ is denoted by $\mathcal{S}(\mathbb{B})$.\vspace{1.4mm}

A biholomorphic mapping on $\mathbb{B}$ is said to be starlike if $f(\mathbb{B})$ is starlike with
respect to the origin. If $\mathbb{B}$ is the unit ball of $X$, a complex Banach space, we let $\mathcal{S}^*(\mathbb{B})$ denote the subset of $\mathcal{S}(\mathbb{B})$ consisting of normalized starlike mappings from $\mathbb{B}$ into $X$. We say that $f$ is
\emph{convex} on $\Omega$ if $f$ is biholomorphic on $\Omega$ and
$f(\Omega)$ is a convex domain in $Y$.\vspace{1.6mm}

In 1998, Sheng Gong \cite{Gong-1999} posed the following conjecture.
\begin{conjA}\cite{Gong-1999} If $f:\mathbb{U}^n\to \mathbb{C}^n$ is a normalized biholomorphic starlike mapping, then
\begin{align*}
\frac{\|D^mf(0)(z^m)\|}{m!}\leq m\|z\|^m,\quad z\in\mathbb{U}^n,\quad m=2,3,\ldots
\end{align*} 
\end{conjA}

We stress that Conjecture~A which we refer to as the Bieberbach–Sheng Conjecture, is not the literal $n$-dimensional analogue of the Bieberbach conjecture -- that literal analogue was already disproved by
Roper and Suffridge \cite{Roper-Suffridge-1999} even at $m=2$. Conjecture~A instead proposes a different, dimension-corrected bound and it is only in
this precise form that the case $m=2$ has been confirmed
(see \cite{Gong-1999}). For related work, we refer the reader to \cite{Gong-1999}-\cite{Graham-2016}.

In several complex variables, classical one-dimensional results often fail to hold in their original form. Most notably, Cartan showed that the Bieberbach conjecture does not hold on the unit ball. Consequently, extending univalent function theory to higher dimensions necessitates the introduction of suitable sub-classes or additional structural constraints.
Recently, coefficient bounds for various sub-classes of normalized biholomorphic mappings in higher dimensions have been established by several authors; see, for example,  Bracci et al. \cite{Bracci-2016}, Bracci and Roth \cite{Bracci-2018}, Graham et al. \cite{Graham-2002,Graham-2016}, Hamada and Honda \cite{Hamada-Honda-2008} and Hamada and Kohr \cite{Hamada-Kohr-2020}.\vspace{1.2mm}

Now, we recall the definition of starlikeness of order $\alpha$ in complex Banach space.
\begin{defiB}\emph{\cite[Definition 1.2]{Hamada-2001}}
Let $f:\mathbb{B}\to X$ be a normalized locally biholomorphic mapping and let $\alpha\in [0,1)$. We say that $f$ is starlike of order $\alpha$ if 
\begin{align*}
\displaystyle
\Re\left\lbrace ||z||/l_z([Df(z)]^{-1}f(z))\right\rbrace>\alpha,\quad z\in\mathbb{B}\backslash\{0\},\quad l_z\in T(z).
\end{align*}
\end{defiB}

Let $\mathcal{S}^*_{\alpha}(\mathbb{B})$ denote the class of starlike mappings of order $\alpha$ on $\mathbb{B}$. When $\alpha=0$, the class $\mathcal{S}^*_{\alpha}(\mathbb{B})$ reduces to the standard starlike class $\mathcal{S}^*(\mathbb{B})$.
There are now many references for starlike mappings in higher dimensions. For more information, see the recent monographs by Gong  \cite{Gong} and Kohr \& Liczberski  \cite{Kohr- Liczberski-1998}. Some of the foundational and original papers in this area include 
\cite{Chu}, \cite{SSQ_CHMS_1990}, \cite{Su2}-\cite{Thomas-Tuneski} and \cite{Xu-He-Xu}.

\begin{defiC}\cite[Definition 1.1]{Liu-Liu-2010}
A normalized locally biholomorphic mapping $f\in\mathcal{H}(\mathbb{B})$ is said to be quasi-convex of type $B$ and order $\alpha\;(0\leq \alpha<1)$ if
\begin{align*}
\displaystyle
\Re\left\lbrace l_z\left([Df(z)]^{-1}\left(D^2f(z)(z,z)+Df(z)z\right)\right)\right\rbrace>\alpha ||z||,\quad z\in\mathbb{B}\backslash\{0\}.
\end{align*}
\end{defiC}

If $\mathbb{B}$ is the unit ball of $X$, we let $\mathcal{K}^{\alpha}_B(\mathbb{B})$ denote the subset of $\mathcal{S}(\mathbb{B})$
consisting of normalized quasi-convex mappings of type $B$ and order $\alpha$ from $\mathbb{B}$ into $X$. When $\alpha=0$, the class $\mathcal{K}^{\alpha}_B(\mathbb{B})$ reduces to $\mathcal{K}_B(\mathbb{B})$, the class of quasi-convex mappings of type $B$. 
When $X=\mathbb{C}$, Definition C is the same as the definition of the normalized univalent convex function of order $\alpha$, which is denoted by $\mathcal{K}_{\alpha}(\mathbb{U})$.
Convex mappings in higher dimensions have been extensively studied in recent literature. For foundational work on this topic, see \cite{Ham4}, \cite{Roper-Suffridge-1999}, \cite{Su2}, \cite{Su3}, \cite{Xu-Dl-Li} and \cite{Zhang-Liu-2002}.
\vspace{1.2mm}

In this paper, by considering biholomorphic mappings
$F:\mathbb B\to X$ of the form $F(z)=f(z)z,\; z\in\mathbb B,$ where $f:\mathbb B\to\mathbb C$ is holomorphic and $f(0)=1$, we obtain the following Banach-space analogue of Robertson's starlike coefficient estimate and the corresponding result for quasi-convex mappings of type $B$ as follows:

 \begin{theo}\label{Th-21}Let $f\in\mathcal{H}(\mathbb{B},\mathbb{C})$ with $f(0)=1$ and $F(z)=zf(z)\in\mathcal{S}^*_{\alpha}(\mathbb{B})$. Then 
 \begin{align}\label{Th-2.1}
 	\left|\frac{l_{z}(D^{k}F(0)(z^{k}))}{k!||z||^{k}}\right|\leq \frac{1}{(k-1)!}\prod\limits_{m=2}^{k}(m-2\alpha),\quad z\in\mathbb{B}\backslash\{0\}
 \end{align}
 for $k=2,3,\ldots$. These estimates are sharp.	
 \end{theo}

 \begin{theo}\label{Th-22}Let $f\in\mathcal{H}(\mathbb{B},\mathbb{C})$ with $f(0)=1$ and $F(z)=zf(z)\in\mathcal{K}^{\alpha}_{B}(\mathbb{B})$. Then 
 \begin{align}\label{Th-2.2}
 \left|\frac{l_{z}(D^{k}F(0)(z^{k}))}{k!||z||^{k}}\right|\leq \frac{1}{k!}\prod\limits_{m=2}^{k}(m-2\alpha),\quad z\in\mathbb{B}\backslash\{0\}
 \end{align}
 for $k=2,3,\ldots$. These estimates are sharp.	
 \end{theo}

\begin{figure}[htbp]
	\centering
	\begin{tikzpicture}
		\begin{axis}[
			width=9.0cm,
			height=6.2cm,
			xmin=1.7, xmax=8.3,
			ymin=0.25, ymax=3.45,
			xlabel={$k$},
			ylabel={Sharp coefficient estimate},
			xtick={2,3,4,5,6,7,8},
			ytick={0.5,1,1.5,2,2.5,3},
			grid=both,
			grid style={dashed, opacity=0.35},
			major grid style={opacity=0.45},
			minor grid style={opacity=0.18},
			minor tick num=1,
			axis line style={thick},
			tick style={thick},
			xlabel style={font=\small},
			ylabel style={font=\small},
			tick label style={font=\small},
			legend style={
				at={(1.04,0.50)},
				anchor=west,
				draw=black!55,
				rounded corners=2pt,
				fill=white,
				fill opacity=0.96,
				text opacity=1,
				font=\footnotesize,
				inner sep=5pt,
				row sep=3pt,
				cells={anchor=west}
			},
			legend cell align={left},
			clip=false
			]
			
			\addplot[
			very thick,
			blue,
			mark=*,
			mark size=2.4pt
			]
			coordinates {
				(2,1.5)
				(3,1.875)
				(4,2.1875)
				(5,2.46094)
				(6,2.70703)
				(7,2.93262)
				(8,3.14209)
			};
			\addlegendentry{
				$B_k^{*}(\alpha)$: starlike
				(Theorem~\ref{Th-21})
			}
			
			\addplot[
			very thick,
			red,
			mark=square*,
			mark size=2.3pt
			]
			coordinates {
				(2,0.75)
				(3,0.625)
				(4,0.546875)
				(5,0.492188)
				(6,0.451172)
				(7,0.418945)
				(8,0.392761)
			};
			\addlegendentry{
				$B_k^{qc}(\alpha)$: quasi-convex
				(Theorem~\ref{Th-22})
			}
			
			\node[
			draw=black!45,
			rounded corners=2pt,
			fill=white,
			fill opacity=0.94,
			text opacity=1,
			font=\scriptsize,
			align=center,
			inner sep=4pt
			] at (axis cs:5.0,3.15)
			{$B_k^{*}(\alpha)=kB_k^{qc}(\alpha)$};
			
		\end{axis}
	\end{tikzpicture}
	
	\caption{Sharp coefficient bounds $B_k^{*}(\alpha)=\frac{1}{(k-1)!}\prod_{m=2}^{k}(m-2\alpha)$
		and $B_k^{qc}(\alpha)=B_k^{*}(\alpha)/k$ at $\alpha=1/4$, for $k=2,\dots,8$.
		The starlike bound increases with $k$ while the quasi-convex bound decreases,
		making the exact factor-of-$k$ relationship $B_k^{qc}=B_k^{*}/k$ visually explicit.}
	\label{fig:coefficient-comparison}
\end{figure}
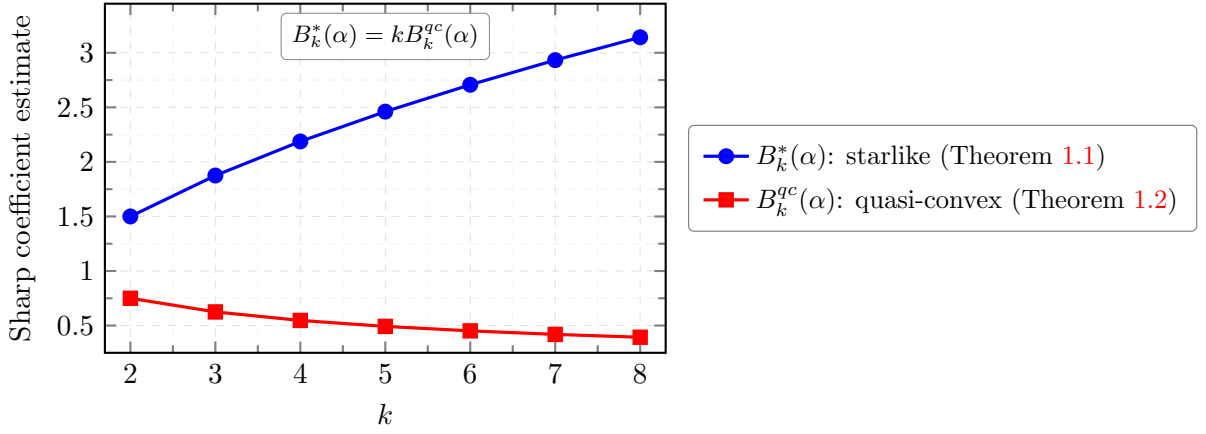

\begin{figure}[htbp]
	\centering
	\begin{tikzpicture}
		\begin{axis}[
			width=9.0cm,
			height=6.4cm,
			xmin=0, xmax=1,
			ymin=0, ymax=6.2,
			xlabel={Order $\alpha$},
			ylabel={Sharp coefficient estimate $B_k^{*}(\alpha)$},
			xtick={0,0.2,0.4,0.6,0.8,1.0},
			ytick={0,1,2,3,4,5,6},
			minor x tick num=1,
			minor y tick num=1,
			grid=both,
			major grid style={
				dashed,
				line width=0.45pt,
				opacity=0.50
			},
			minor grid style={
				dotted,
				line width=0.30pt,
				opacity=0.38
			},
			legend style={
				at={(1.04,0.50)},
				anchor=west,
				draw=black!55,
				rounded corners=2pt,
				fill=white,
				font=\footnotesize
			},
			legend cell align={left},
			clip=false
			]
			
			\addplot[
			very thick,
			blue,
			domain=0:1,
			samples=100
			]{2-2*x};
			\addlegendentry{$k=2$}
			
			\addplot[
			very thick,
			red,
			domain=0:1,
			samples=100
			]{(2-2*x)*(3-2*x)/2};
			\addlegendentry{$k=3$}
			
			\addplot[
			very thick,
			orange,
			domain=0:1,
			samples=100
			]{(2-2*x)*(3-2*x)*(4-2*x)/6};
			\addlegendentry{$k=4$}
			
			\addplot[
			very thick,
			teal,
			domain=0:1,
			samples=100
			]{(2-2*x)*(3-2*x)*(4-2*x)*(5-2*x)/24};
			\addlegendentry{$k=5$}
			
			\addplot[
			very thick,
			violet,
			domain=0:1,
			samples=100
			]{(2-2*x)*(3-2*x)*(4-2*x)*(5-2*x)*(6-2*x)/120};
			\addlegendentry{$k=6$}
			
			\addplot[
			only marks,
			mark=*,
			mark size=2pt,
			black
			] coordinates {(0.5,1)};
			
			\node[
			draw=black!45,
			rounded corners=2pt,
			fill=white,
			font=\scriptsize,
			align=center,
			inner sep=3pt
			] (crossing)
			at (axis cs:0.55,2.45)
			{all curves meet\\at $(\tfrac12,1)$};
			
			\draw[
			-{Stealth[length=2.2mm]},
			black!65,
			line width=0.7pt
			]
			(crossing.south west)
			.. controls (axis cs:0.62,1.95)
			and (axis cs:0.55,1.45) ..
			(axis cs:0.5,1);
			
			\node[
			draw=black!45,
			rounded corners=2pt,
			fill=white,
			font=\scriptsize,
			align=center,
			inner sep=4pt
			] at (axis cs:0.86,1.25)
			{$B_k^{*}(\alpha)\to0$\\as $\alpha\to1^{-}$};
			
		\end{axis}
	\end{tikzpicture}
	
	\caption{Dependence of the sharp coefficient estimate
		$B_k^{*}(\alpha)$ on the order $\alpha$ for
		$k=2,3,4,5,6$. For each fixed $k$, $B_k^{*}(\alpha)$
		decreases with $\alpha$ and tends to zero as
		$\alpha\to1^{-}$; all curves pass through the common
		point $(1/2,1)$.}
	
	\label{fig:alpha-dependence}
\end{figure}
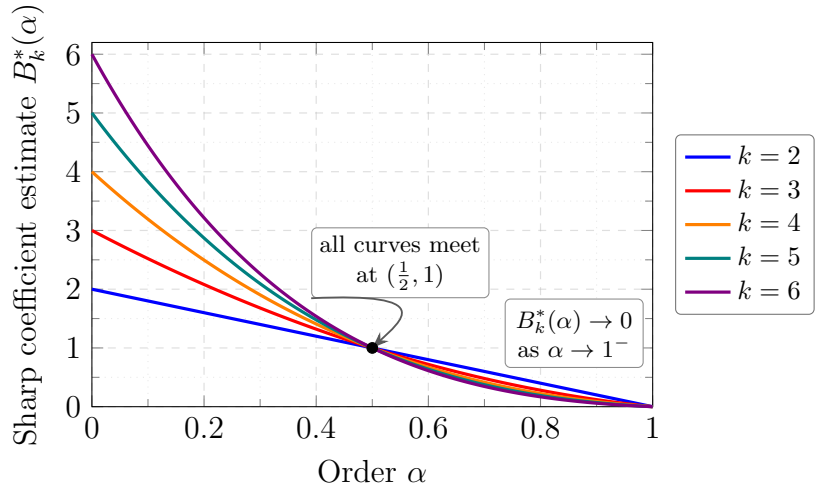

\section{\bf{Proof of Theorem \ref{Th-21}}}
Let $\mathcal{P}$ denote the class of analytic functions $p$ in $\mathbb{U}$ satisfying $p(0) = 1$ and $\Re\{p(z)\}> 0$ for all $z \in \mathbb{U}$. The lemma below is instrumental in establishing our main results.
\begin{lem}\label{Lem1} {\em{\cite{Cho-Kumar-Ravichandran-2019}}}
	Let $p(z)=1+\sum_{n=1}^{\infty}p_nz^n\in\mathcal{P}$. Then $|p_n|\leq 2$ and $|p_n-p_mp_{n-m}|\leq 2$, $n\geq 2$, $1\leq m\leq n-1$.
\end{lem}

\begin{proof}[Proof of Theorem \ref{Th-21}]
We have $f\in\mathcal{H}(\mathbb{B},\mathbb{C})$ with $f(0)=1$ and $F(z)=zf(z)\in\mathcal{S}^*_{\alpha}(\mathbb{B})$. Let $z\in \mathbb{B}\backslash\{0\}$ be fixed. Suppose $z_0=z/\|z\|$. Define $g:\mathbb{U}\to \mathbb{C}$ as follows
\begin{align}\label{Th-21.0}	
\displaystyle
g(\zeta)=\begin{cases}
\frac{\zeta}{l_{z_0}\left([D F(\zeta z_0)]^{-1}F(\zeta z_0)\right)},& \zeta\neq 0,\\[8pt]
1& \zeta=0.
\end{cases}
\end{align}

Clearly $g\in\mathcal{H}(\mathbb{U})$ and $g(0)=1$.
Since $l_{\zeta z}(.)=\frac{|\zeta|}{\zeta}l_z$ for each $\zeta\in\mathbb{C}\backslash\{0\}$, we have
\begin{align*}
g(\zeta)=\frac{\zeta}{l_{z_0}\left([D F(\zeta z_0)]^{-1}F(\zeta z_0)\right)}=\frac{\|\zeta z_0\|}{l_{\zeta z_0}\left([D F(\zeta z_0)]^{-1}F(\zeta z_0)\right)}
\end{align*}	
for $\zeta\in\mathbb{U}\backslash\{0\}$  and so  $\Re\{g(\zeta)\}>\alpha$
for all $\zeta\in\mathbb{U}$. Suppose
\begin{align}\label{Th-21.a}
g(\zeta)=1+g_1\zeta+g_2\zeta^2+g_3\zeta^3\ldots
\end{align}

Let $p(\zeta)=(g(\zeta)-\alpha)/(1-\alpha)$
for all $\zeta\in\mathbb{U}$. Clearly $p\in\mathcal{H}(\mathbb{U})$ with $\Re\{p(\zeta)\}>0$ for all $\zeta\in\mathbb{U}$ and $p(0)=1$. 
Suppose $h(\zeta)=1+p_1\zeta+p_2\zeta^2+p_3\zeta^3+\ldots$.
According to Lemma \ref{Lem1}, we have $|p_k|\leq 2$ for $k=1,2,\ldots$.
Clearly 
\begin{align*}
g(\zeta)=1+\sum\limits_{i=1}^{\infty}g_i\zeta^i=1+(1-\alpha)p_1\zeta+(1-\alpha)p_2\zeta^2+(1-\alpha)p_3\zeta^3+\ldots
\end{align*}

Consequently, for $k=1,2,\ldots$, we have
\begin{align}\label{Th-21.b}
|g_k|\leq (1-\alpha)|p_k|\leq 2(1-\alpha).
\end{align}

Using a similar method as in Theorem 7.1.14 \cite{Graham-Kohr}, we have
\begin{align}\label{Th-21.2}
	[D F(z)]^{-1}\eta=\frac{1}{f(z)}\left\lbrack \eta-\frac{\frac{(Df(z)\eta)z}{f(z)}}{1+\frac{Df(z)z}{f(z)}}\right\rbrack.
\end{align}

Therefore 
\begin{align}\label{Th-21.2}
	[D F(z)]^{-1}F(z)=\frac{zf(z)}{f(z)+D f(z)z},
\end{align}
which implies that $\frac{\|z\|}{l_z\left([D F(z)]^{-1}F(z)\right)}=1+\frac{Df(z) z}{f(z)}$
and so
\begin{align}\label{Th-21.3}
	g(\zeta)=\frac{\|\zeta z_0\|}{l_{\zeta z_0}\left([D F(\zeta z_0)]^{-1}F(\zeta z_0)\right)}=1+\frac{Df(\zeta z_0) \zeta z_0}{f(\zeta z_0)}.
\end{align}

Let 
\begin{align}\label{Th-21.4}
f(z)=1+\sum\limits_{k=1}^{\infty} \frac{1}{k!}D^kf(0)(z^k).
\end{align}	

Clearly
\begin{align}\label{Th-21.5}
Df(z)(z)=\sum\limits_{k=1}^{\infty} \frac{k}{k!}D^kf(0)(z^k).
\end{align}	

If we take $ h(\zeta)=f(\zeta z_0)$, then from \eqref{Th-21.4}, we have
\begin{align}\label{Th-21.6}
\displaystyle
 h(\zeta)=f(\zeta z_0)=1+a_1\zeta+a_2\zeta^2+\ldots.
\end{align}
where
\begin{align}\label{Th-21.7}
a_k=\frac{1}{k!}D^k f(0)(z_0^k)
\end{align}
for $k=1,2,\ldots$
We see that \eqref{Th-21.5} yields
\begin{align}\label{Th-21.8}
\displaystyle
D f(\zeta z_0)(\zeta z_0)=\sum\limits_{k=1}^{\infty}ka_k\zeta^k
\end{align}

Obviously \eqref{Th-21.6} and \eqref{Th-21.8}, yields $ h'(\zeta)\zeta=D f(\zeta z_0)(\zeta z_0)$. So from \eqref{Th-21.3}, we get
\begin{align}\label{Th-21.9}
1+\frac{\zeta h'(\zeta)}{h(\zeta)}=g(\zeta)=\sum\limits_{k=0}^{\infty}g_k\zeta^k,
\end{align}
where $g_0=1$.
Note that 
\begin{align*}
\frac{1+(1-2\alpha)\zeta}{1-\zeta}=1+2(1-\alpha)\zeta(1-\zeta)^{-1}=1+2(1-\alpha)\sum\limits_{k=1}^{\infty}\zeta^k=\sum\limits_{k=0}^{\infty}b_k\zeta^k,
\end{align*}
where $b_0=1$, $b_k=2(1-\alpha)$, $k\in\mathbb{N}$. Now \eqref{Th-21.b} yields $|g_k|\leq b_k$, $k\in\mathbb{N}\cup\{0\}$. Hence
\begin{align*}
g(\zeta)\ll\frac{1+(1-2\alpha)\zeta}{1-\zeta}
\end{align*}
and so \eqref{Th-21.9} yields
\begin{align*}
1+\frac{\zeta h'(\zeta)}{h(\zeta)}\ll \frac{1+(1-2\alpha)\zeta}{1-\zeta}.
\end{align*}

Consequently 
\begin{align}\label{Th-21.10}
h(\zeta)\ll \phi(\zeta):=\frac{1}{(1-\zeta)^{2(1-\alpha)}}=1+\sum\limits_{k=1}^{\infty}c_k\zeta^k,
\end{align}
where
\begin{align}\label{Th-21.11}
c_k=\frac{(2-2\alpha)(3-2\alpha)\ldots((k+1)-2\alpha)}{k!}.
\end{align}
Using \eqref{Th-21.6} to \eqref{Th-21.10}, we get $|a_k|\leq c_k$, $k\in\mathbb{N}$. Consequently \eqref{Th-21.7} and \eqref{Th-21.11} give
\begin{align}\label{Th-21.12}
\left|\frac{1}{k!}D^k f(0)(z_0^k)\right|\leq \frac{(2-2\alpha)(3-2\alpha)\ldots((k+1)-2\alpha)}{k!}=\frac{1}{k!}\prod\limits_{m=2}^{k+1}(m-2\alpha)
\end{align}
for $k=1,2,3,\ldots$.
Since $F(z)=zf(z)$, we have (see \cite[Theorem 2]{Xiong-Yu-Wang-2025})
\begin{align}\label{Th-21.13}
D^mF(0)(z^m)=mD^{m-1}f(0)(z^{m-1})z
\end{align}
where $m\in\mathbb{N}$. As $z_0=z/||z||$, we deduce from \eqref{Th-21.13} that
\begin{align}\label{Th-21.14}
\displaystyle
\frac{l_{z_0}(D^mF(0)(z_0^m))}{m!}=\frac{D^{m-1}f(0)(z_0^{m-1})}{(m-1)!}.
\end{align}

Finally from \eqref{Th-21.12} and \eqref{Th-21.14}, we have
\begin{align*}
\left|\frac{l_{z_0}(D^{k+1}F(0)(z_0^{k+1}))}{(k+1)!}\right|\leq \frac{1}{k!}\prod\limits_{m=2}^{k+1}(m-2\alpha),
\end{align*}
i.e., for $k=2,3,\ldots$, we get the desired inequalities \eqref{Th-2.1}.

To establish the sharpness of the bound, let us consider the function $F$ defined by
\begin{align*}
F(z)=zf(z)=\frac{z}{\left(1-l_u(z)\right)^{2(1-\alpha)}},\quad z\in\mathbb{B},\quad \|u\|=1.
\end{align*}

According to Lemma 2.2 \cite{Xu-He-Xu}, obviously $F\in\mathcal{S}^*(\mathbb{B})$. Note that 
\begin{align*}
f(z)=1+\sum_{k=1}^{\infty}\frac{(2-2\alpha)(3-2\alpha)\ldots((k+1)-2\alpha)}{k!}(l_u(z))^{k},\quad z\in\mathbb{B}.
\end{align*}

Clearly \eqref{Th-21.13} yields
\begin{align*}
\frac{l_z(D^{k+1}F(0)(z^{k+1}))}{(k+1)!}=(l_u(z))^k\|z\|\;\frac{1}{k!}\prod\limits_{m=2}^{k+1}(m-2\alpha).
\end{align*}

Taking $z=ru\;(0<r<1)$, we get for $k=2,3,\ldots$
\begin{align*}
\left|\frac{l_z(D^{k}F(0)(z^{k}))}{k!||z||^{k}}\right|=\frac{1}{(k-1)!}\prod\limits_{m=2}^{k}(m-2\alpha),\quad z\in\mathbb{B}\backslash\{0\}.
\end{align*}

\end{proof}

\section{\bf{Proof of Theorem \ref{Th-22}}}
\begin{proof}
By the given condition, we have $F\in\mathcal{K}^{\alpha}_{B}(\mathbb{B})$ and so 
\begin{align}\label{Lem1.1}
\displaystyle
\Re\left\lbrace \frac{l_{z_0}\left([DF(z)]^{-1}\phi(z)\right)}{||z||}\right\rbrace>\alpha,
\end{align}
$z\in\mathbb{B}\backslash\{0\}$ and $l_z\in T(z)$, where $\phi(z)=D^2 F(z)(z^2)+DF(z) z$.	
Let $z\in \mathbb{B}\backslash\{0\}$ be fixed and let $z_0=z/\|z\|$.
We define the function $H:\mathbb{U}\to \mathbb{C}$ as follows
\begin{align}\label{Th2.2}	
\displaystyle
H(\zeta)=\begin{cases}
\frac{l_{z_0}\left([DF(\zeta z_0)]^{-1}\phi(\zeta z_0)\right)}{\zeta},& \zeta\neq 0,\\[8pt]
1& \zeta=0.
\end{cases}
\end{align}

Clearly $H\in\mathcal{H}(\mathbb{U})$ and $H(0)=1$.
Since $l_{\zeta z}(.)=\frac{|\zeta|}{\zeta}l_z$ for each $\zeta\in\mathbb{U}\backslash\{0\}$, we have
\begin{align}\label{Th2.3}
\displaystyle
H(\zeta)=\frac{l_{\zeta z_0}\left([DF(\zeta z_0)]^{-1}\phi(\zeta z_0)\right)}{||\zeta z_0||}
\end{align}	
and so $\Re\{H(\zeta)\}>\alpha$ for all $\zeta\in\mathbb{U}$.
Suppose
\begin{align}\label{Th2.3a}
H(\zeta)=1+H_1\zeta+H_2\zeta^2+H_3\zeta^3+\ldots
\end{align}

Clearly 
\begin{align}\label{Th2.4}
|H_k|\leq 2(1-\alpha)
\end{align}
for $k=1,2,\ldots$.
Using a similar method as in Theorem 7.1.14 \cite{Graham-Kohr}, we have
\begin{align}\label{Th2.5}
\displaystyle
[D F(z)]^{-1}\phi(z)=\frac{\phi(z)}{f(z)+D f(z)z}.
\end{align}

As $F(z)$ is normalized, it has the power series expansion
\begin{align}\label{Th2.9a}
\displaystyle
F(z)=z+\sum\limits_{k=2}^{\infty}\frac{1}{k!}D^k F(0)(z^k).
\end{align}

Clearly
\begin{align}\label{Th2.9}
\displaystyle	
D F(z)=I+\sum\limits_{k=2}^{\infty}\frac{k}{k!}D^k F(0)(z^{k-1})
\end{align}
and 
\begin{align}\label{Th2.10}
\displaystyle
D^2 F(z)=\sum\limits_{k=2}^{\infty}\frac{k(k-1)}{k!}D^k F(0)(z^{k-2}).
\end{align}

Using \eqref{Th2.9} and \eqref{Th2.10}, we get
\begin{align}\label{Th2.11}
\displaystyle
D^2 F(z)(z^2)+DF(z)
z=&z+\sum\limits_{k=2}^{\infty}\frac{k^2}{k!}D^k F(0)(z^k).
\end{align}

Now substituting \eqref{Th-21.13} into \eqref{Th2.11}, we get
\begin{align}\label{Th2.13}
\displaystyle
D^2 F(z)(z^2)+DF(z)
z=&z+\sum\limits_{k=2}^{\infty}\frac{k^3}{k!}D^{k-1} f(0)(z^{k-1}) z.
\end{align}

Clearly 
\begin{align}\label{Th2.13a}
\displaystyle
\phi(\zeta z_0)=\zeta z_0\left(1+\sum\limits_{m=1}^{\infty}\frac{(m+1)^2}{m!}D^{m} f(0)(z^{m})\zeta^m\right).
\end{align}

Let
\begin{align}\label{Th2.14}
\displaystyle
f(\zeta z_0)=1+\sum\limits_{k=1}^{\infty}\frac{1}{k!}D^k f(0)(z_0^k)\zeta^k.
\end{align}

Clearly
\begin{align}\label{Th2.15}
\displaystyle
D f(\zeta z_0)(\zeta z_0)=\sum\limits_{k=1}^{\infty}\frac{k}{k!}D^k f(0)(z_0^k)\zeta^k.
\end{align}

Suppose $g(\zeta)=f(\zeta z_0)\zeta$. Then we have
\begin{align}\label{Th2.7}
g'(\zeta)=f(\zeta z_0)+Df(\zeta z_0)\zeta z_0.
\end{align}

In view of \eqref{Th2.7} and adding \eqref{Th2.9} and \eqref{Th2.10}, we obtain
\begin{align}\label{Th2.16}
\displaystyle
g'(\zeta)=f(\zeta z_0)+D f(\zeta z_0)\zeta z_0=1+\sum\limits_{k=1}^{\infty}\frac{k+1}{k!}D^k f(0)(z_0^k)\zeta^k.
\end{align}	

Clearly
\begin{align}\label{Th2.17}
\displaystyle
\zeta g''(\zeta)+g'(\zeta)=1+\sum\limits_{k=1}^{\infty}\frac{(k+1)^2}{k!}D^k f(0)(z_0^k)\zeta^k.
\end{align}

Now from \eqref{Th2.5}, \eqref{Th2.13a} and \eqref{Th2.16}, we get
\begin{align}\label{Th2.18}
\displaystyle
\frac{l_{\zeta z_0}\left([D F(\zeta z_0)]^{-1}\phi(\zeta z_0)\right)}{||\zeta z_0||}=\frac{1+\sum\limits_{m=1}^{\infty}\frac{(m+1)^2}{m!}D^{m} f(0)(z^{m})\zeta^m}{1+\sum\limits_{k=1}^{\infty}\frac{k+1}{k!}D^k f(0)(z_0^k)\zeta^k}.
\end{align}

Combining \eqref{Th2.16} and \eqref{Th2.17} to \eqref{Th2.18}, we obtain
\begin{align}\label{Th2.19}
\displaystyle
\frac{l_{\zeta z_0}\left([D F(\zeta z_0)]^{-1}\phi(\zeta z_0)\right)}{||\zeta z_0||}
=&\frac{\zeta g''(\zeta)+g'(\zeta)}{g'(\zeta)}=1+\zeta\frac{
g''(\zeta)}{g'(\zeta)}.
\end{align}

In view of \eqref{Th2.3} and \eqref{Th2.3a}, we get from \eqref{Th2.19} that
\begin{align}\label{Th2.20}
1+\zeta\frac{g''(\zeta)}{g'(\zeta)}=H(\zeta)=\sum\limits_{k=0}^{\infty}H_k\zeta^k,
\end{align}
where $H_0=1$.
Note that 
\begin{align*}
	\frac{1+(1-2\alpha)\zeta}{1-\zeta}=\sum\limits_{k=0}^{\infty}b_k\zeta^k,
\end{align*}
where $b_0=1$, $b_k=2(1-\alpha)$, $k\in\mathbb{N}$. So \eqref{Th2.4} yields $|H_k|\leq b_k$, $k\in\mathbb{N}\cup\{0\}$. Hence
\begin{align*}
H(\zeta)\ll\frac{1+(1-2\alpha)\zeta}{1-\zeta}
\end{align*}
and so \eqref{Th2.20} yields
\begin{align*}
	1+\frac{\zeta g''(\zeta)}{g'(\zeta)}\ll \frac{1+(1-2\alpha)\zeta}{1-\zeta}.
\end{align*}

Consequently 
\begin{align}\label{Th2.21}
g'(\zeta)\ll \phi(\zeta):=\frac{1}{(1-\zeta)^{2(1-\alpha)}}=1+\sum\limits_{k=1}^{\infty}c_k\zeta^k,
\end{align}
where
\begin{align}\label{Th2.22}
	c_k=\frac{(2-2\alpha)(3-2\alpha)\ldots((k+1)-2\alpha)}{k!}
\end{align}
for $k=1,2,3,\ldots$.
Now using \eqref{Th2.16} to \eqref{Th2.21}, we get 
\begin{align*}
\left|\frac{k+1}{k!}D^k f(0)(z_0^k)\zeta^k\right|\leq c_k
\end{align*}
for $k=1,2,3,\ldots$. Consequently from \eqref{Th2.22}, we have
\begin{align}\label{Th2.23}
\left|\frac{k+1}{k!}D^k f(0)(z_0^k)\right|\leq \frac{(2-2\alpha)(3-2\alpha)\ldots((k+1)-2\alpha)}{k!}=\frac{1}{k!}\prod\limits_{m=2}^{k+1}(m-2\alpha)
\end{align}
for $k=1,2,3,\ldots$.
Finally from \eqref{Th2.23} and \eqref{Th-21.14}, we have
\begin{align*}
	\left|\frac{l_{z_0}(D^{k+1}F(0)(z_0^{k+1}))}{(k+1)!}\right|\leq \frac{1}{(k+1)!}\prod\limits_{m=2}^{k+1}(m-2\alpha),
\end{align*}
i.e., for $k=2,3,\ldots$,  we get the desired inequalities \eqref{Th-2.2}.

To establish the sharpness of the bound, let us consider the function $F$ defined by
\begin{align*}
F(z)=zf(z)=\frac{f_1(l_u(z))}{l_u(z)}z,\quad z\in\mathbb{B},\quad \|u\|=1,
\end{align*}
where
\begin{align*}
f_1(l_u(z))=
\begin{cases}
\frac{1-(1-l_u(z))^{2\alpha-1}}{2\alpha-1}, & \alpha\neq \frac{1}{2}, \\[10pt]
-\log (1-l_u(z)), & \alpha=\frac{1}{2}.
\end{cases}
\end{align*}
 We know that $f_1\in\mathcal{K}_{\alpha}(\mathbb{U})$ (see \cite[Problem 2.3.7]{Graham-Kohr}).
According to Lemma 2.1 \cite{Liu-Liu-2010}, we have $F\in\mathcal{K}^{\alpha}_B(\mathbb{B})$.
Suppose $\alpha\neq 1/2$. In this case, we have
\begin{align*}
f(l_u(z))=\frac{1}{2\alpha-1}\sum\limits_{k=0}^{\infty}(-1)^{k+2}\binom{2\alpha-1}{k+1}(l_u(z))^{k}.
\end{align*}

Clearly \eqref{Th-21.13} yields
\begin{align*}
\frac{l_z(D^{k+1}F(0)(z^{k+1}))}{(k+1)!}=(-1)^k\frac{(2\alpha-2)(2\alpha-3)\ldots(2\alpha-k-1)}{(k+1)!}\;(l_u(z))^k\|z\|.
\end{align*}

Taking $z=ru\;(0<r<1)$, we get for $k=2,3,\ldots$
\begin{align*}
	\left|\frac{l_z(D^{k}F(0)(z^{k}))}{k!||z||^{k}}\right|=\frac{1}{k!}\prod\limits_{m=2}^{k}(m-2\alpha),\quad z\in\mathbb{B}\backslash\{0\}.
\end{align*}

Suppose $\alpha=1/2$. In this case, we have
\begin{align*}
	f(z)=\sum_{k=0}^{\infty}\frac{1}{k+1}(l_u(z))^{k},\quad z\in\mathbb{B}.
\end{align*}

Let us take $z=ru\;(0<r<1)$. Now using \eqref{Th-21.13}, we find for $k=2,3,\ldots$ that
\begin{align*}
	\left|\frac{l_z(D^{k}F(0)(z^{k}))}{k!||z||^{k}}\right|=\frac{1}{k},\quad z\in\mathbb{B}\backslash\{0\}.
\end{align*}

\end{proof}

\noindent\textbf{Conflict of interest:} The authors declare that there is no conflict  of interest regarding the publication of this paper.\vspace{1.2mm}

\noindent {\bf Funding:} At present, no funding is available for this work.\vspace{1.2mm}

\noindent\textbf{Data availability statement:}  Data sharing not applicable to this article as no datasets were generated or analysed during the current study.\vspace{1.2mm}


\end{document}